\documentclass[12pt]{article}

\usepackage[paperwidth=8.5in, paperheight=11in, top=1in, bottom=1in, left=.5in, right=.5in]{geometry}

\usepackage{amsmath,amsfonts,amssymb,epsfig,euscript,mathrsfs,graphicx}
\usepackage{xspace,multicol,enumitem}
\usepackage{theorem}
\usepackage{color}
\usepackage{caption, float}
\usepackage[dvipsnames]{xcolor}

\newtheorem{theorem}{Theorem}[section]

\newtheorem{question}{Question}

\newcommand{\ds}{\displaystyle}
\newcommand{\q}{\mathbb{Q}}
\newcommand{\z}{\mathbb{Z}}

\newenvironment{proof}[1][Proof]{\begin{trivlist}
\item[\hskip \labelsep {\bfseries #1}]}{\end{trivlist}}

\newcommand{\qed}{\hfill \ensuremath{\Box}}

\begin{document}

\thispagestyle{empty}
\title{\textbf{A Triangle-free, 4-chromatic $\mathbb{Q}^3$ Euclidean Distance Graph Scavenger Hunt!}}

\author{\textbf{Jonathan Joe\footnote{jonathan.joe@mga.edu} \hspace{20pt}Matt Noble\footnote{matthew.noble@mga.edu}}  \\
Department of Mathematics and Statistics\\
Middle Georgia State University\\
Macon, GA 31206}

\date{}
\maketitle

\begin{abstract}
For $d > 0$, define $G(\mathbb{Q}^3, d)$ to be the graph whose set of vertices is the rational space $\mathbb{Q}^3$, where two vertices are adjacent if and only if they are a Euclidean distance $d$ apart.  Let $\chi(\mathbb{Q}^3, d)$ be the chromatic number of such a graph or, in other words, the minimum number of colors needed to color the points of $\mathbb{Q}^3$ so that no two points at distance $d$ apart receive the same color.  An open problem, originally posed by Benda and Perles in the 1970s, asks if there exists $d$ such that $\chi(\mathbb{Q}^3, d) = 3$.  Through numerous efforts over the years, $\chi(\mathbb{Q}^3, d)$ has been determined for many values of $d$, and for all those distances $d$ where $\chi(\mathbb{Q}^3, d)$ has not been exactly pinned down, it is known that $\chi(\mathbb{Q}^3, d) \in \{3,4\}$.  In our work, we detail several search algorithms we have employed to find $4$-chromatic subgraphs of various graphs $G(\mathbb{Q}^3, d)$ whose chromatic number was previously unknown.  Ultimately, we conjecture that no $3$-chromatic $G(\mathbb{Q}^3, d)$ exists.  Along the way, we pose a few related questions that we feel are of interest in their own right.\\

\end{abstract}

\section{Introduction}

Let $\mathbb{R}$, $\q$, and $\z$ denote the rings of real numbers, rational numbers, and integers, respectively.  As is typical, the set $\mathbb{Q}^n$ -- that is, the set of all $n$-tuples whose entries are in $\mathbb{Q}$ -- will be referred to as the set of \textit{rational points} of $\mathbb{R}^n$.  For points $a,b \in \mathbb{R}^n$, designate by $|a-b|$ the Euclidean distance from $a$ to $b$.  

The concept central to this work will be that of the \textit{Euclidean distance graph}, with a detailed historybeing found in \cite{soifer}.  For $d > 0$ and $S \subseteq \mathbb{R}^n$, define $G(S, d)$ to be the graph whose vertices are the points of $S$, with any two vertices $a,b$ being adjacent if and only if $|a-b|=d$.  Define the graph $G(S, d)$ to be \textit{non-trivial} if $d$ is actually realized as a distance between points of $S$, as otherwise, $G(S, d)$ would have an empty edge set and would not be of interest.  Let $\chi(S, d)$ be the chromatic number of $G(S, d)$ or, in other words, the minimum number of colors needed to color the points of $S$ so that no two points distance $d$ apart receive the same color.  Such a coloring of $G(S, d)$ is often referred to as being \textit{proper} or is said to \textit{forbid} the distance $d$.

The following was originally posed by Benda and Perles in \cite{bendaperles}.

\begin{question} \label{mainquestion} Does there exist $d > 0$ such that $\chi(\q^3, d) = 3$?
\end{question}

Before we delve into what is known about Question \ref{mainquestion}, and what the general thrust of our current work will be, some time should be spent detailing the state of knowledge concerning proper colorings of distance graphs on the rational points, along with the unusual history of \cite{bendaperles} itself.  Although the notion of properly coloring $\mathbb{R}^n$ was initially put forth by Edward Nelson in the early 1950s, Woodall \cite{woodall} was the first to consider chromatic numbers of graphs $G(\mathbb{Q}^n, d)$ by showing as a secondary result in a 1972 article \cite{woodall} that $\chi(\mathbb{Q}^2, 1) = 2$.  Benda and Perles produced their manuscript in the mid-1970s, but were unaware of \cite{woodall}, as they give an alternate proof that $\chi(\mathbb{Q}^2, 1) = 2$ along with proofs showing $\chi(\mathbb{Q}^3, 1) = 2$ and $\chi(\mathbb{Q}^4, 1) = 4$.  They also note that $\chi(\mathbb{Q}^3, d_1)$ and $\chi(\mathbb{Q}^3, d_2)$ may be unequal for distinct $d_1, d_2$ by displaying a 4-chromatic subgraph of $G(\mathbb{Q}^3, \sqrt2)$.  Benda and Perles appear to be aware in \cite{bendaperles} of the fact that $\chi(\mathbb{Q}^4, d) = 4$ for all non-trivial graphs $G(\mathbb{Q}^4, d)$, which in regard to Question \ref{mainquestion} implies that $\chi(\mathbb{Q}^3, d) \leq 4$ for all $d > 0$, but they do not give a formal proof.

Despite its original results and the questions posed in \cite{bendaperles}, Benda and Perles surprisingly declined to  publish their manuscript, and it did not formally appear in print until the year 2000.  Perhaps even more surprising, given that in the pre-internet days of the 1970s and -80s, even published works were often slow to make their way to the greater mathematical landscape, by the mid-1980s \cite{bendaperles} had found its way to the growing community of mathematicians interested in Euclidean distance graph coloring problems, and its results had become widely known (see \cite{bendaperleshistory} for a historical perspective).          

At present Question \ref{mainquestion} has not been fully resolved, but through numerous efforts over the years, $\chi(\q^3, d)$ is known for many values of $d$.  To begin, note that when $d_1, d_2$ are rational multiples of each other, the graphs $G(\mathbb{Q}^3, d_1)$ and $G(\mathbb{Q}^3, d_2)$ are isomorphic by an obvious scaling argument, and it follows that any non-trivial graph $G(\mathbb{Q}^3, d)$ is isomorphic to a graph $G(\mathbb{Q}^3, \sqrt{r})$ where $r$ is a square-free positive integer.  Indeed, for the rest of our work, it is assumed that any graph $G(\mathbb{Q}^3, \sqrt{r})$ has $r$ of this form.  In \cite{twocolors}, Johnson shows that if $r$ is odd, $G(\mathbb{Q}^3, \sqrt{r})$ can be properly 2-colored.  In \cite{chow}, Chow proves that for $r$ even, $\chi(\mathbb{Q}^3, \sqrt{r}) \geq 3$.  As previously mentioned, it is hinted at in \cite{bendaperles} that $\chi(\mathbb{Q}^3, \sqrt{r}) \leq 4$ for all $r$, and we note that this fact is seen with proof in \cite{burkert}, \cite{jst}, and likely in other places as well.  It is shown by the second author in \cite{noble1} that if $r$ is even, but has no odd prime factors congruent to $2$ modulo $3$, then $\chi(\mathbb{Q}^3, \sqrt{r}) = 4$.  For all remaining values of $r$ -- that is, those $r$ which are even, and have at least one odd prime factor congruent to $2$ modulo $3$ -- $\chi(\mathbb{Q}^3, \sqrt{r})$ is unknown, except for the specific case of $r = 10$ where it is shown in \cite{noble1} that $\chi(\mathbb{Q}^3, \sqrt{10}) = 4$.

Let  $T = \{10, 22, 30, 34, 46, 58, 66, \ldots \}$ be the set of all square-free, even positive integers, each of which contains at least one odd prime factor congruent to $2$ modulo $3$.  From the discussion above, we have that $\chi(\mathbb{Q}^3, \sqrt{10}) = 4$ and for all other $t \in T$, $\chi(\mathbb{Q}^3, \sqrt{t}) \in \{3,4\}$.  Given the progression of results listed in the previous paragraph, one may naturally wonder what it is about those $t \in T$ that has made $\chi(\mathbb{Q}^3, \sqrt{t})$ hard to exactly pin down.  The difficulty could stem from one of two reasons, depending on whether $\chi(\mathbb{Q}^3, \sqrt{t})$ equals 3 or 4, for that particular value of $t$. Results in the literature giving proper colorings of graphs $G(\mathbb{Q}^n, d)$ for $n \in \{2,3,4\}$ and any $d > 0$ may vary somewhat in their methodology, but at their heart, they all use the same type of argument.  Elementary number-theoretic facts concerning quadratic residues and representations of rational numbers as sums of squares are assembled, and they are then combined, typically via some manner of induction argument, to show the existence of (or explicitly produce) the desired coloring.  If it just so happens that $\chi(\mathbb{Q}^3, \sqrt{t}) = 3$ for some $t \in T$, a proper 3-coloring of $G(\mathbb{Q}^3, \sqrt{t})$ would be a wholly original departure from what has previously been done.  For this reason, we doubt that a 3-chromatic graph $G(\mathbb{Q}^3, \sqrt{t})$ exists.

Now, a reader is certainly free to object and say that just because something has not been seen before, one should not discount its possibility.  Yet still, it seems that the most likely avenue of success would be in attempting to find a 4-chromatic subgraph of the $G(\mathbb{Q}^3, \sqrt{t})$ under consideration.  The difficulty here lies in the fact that $G(\mathbb{Q}^3, \sqrt{t})$ is triangle-free for all $t \in T$ (seen as an immediate consequence of Ionascu's work \cite{ionascu1} enumerating equilateral triangles with vertices in $\mathbb{Z}^3$).  There certainly exist triangle-free 4-chromatic graphs.  In fact, there exist triangle-free graphs of arbitrarily large chromatic number.  However, constructing them as subgraphs of a desired $G(\mathbb{Q}^3, \sqrt{t})$ appears to be quite tricky, and the author of \cite{noble1} will somewhat sheepishly admit to stumbling onto a 4-chromatic subgraph of $G(\mathbb{Q}^3, \sqrt{10})$ mostly by accident.  

In Sections 2, 3, and 4 of this article, we will make our search methods more formal and describe three algorithms we have employed to find 4-chromatic subgraphs of $G(\mathbb{Q}^3, \sqrt{t})$ for various $t$, thus showing $\chi(\mathbb{Q}^3, \sqrt{t}) = 4$ for those values of $t$.  These successes lead us to conjecture in Section 5 that Question \ref{mainquestion} has a negative answer.  Along the way, we will pose a few additional questions related to this search process which we feel are of interest in their own right.

\section{A Greedy Approach}

In a 2003 article \cite{mann}, Mann utilizes several search algorithms to produce subgraphs of $G(\mathbb{Q}^n, 1)$ having large chromatic number, for $n \in \{6,7,8\}$.  The second algorithm described in \cite{mann} essentially follows a greedy approach.  One begins with an induced subgraph of $G(\mathbb{Q}^n, 1)$ having a sufficiently large chromatic number, along with a subset $A$ of $\mathbb{Q}^n$.  Points of $A$ are then iteratively selected to be new vertices in the graph, where at each step, the selected point is chosen that is adjacent to the largest number of vertices already in the graph.  After each new vertex is added, the chromatic number of the resulting graph is computed, and the algorithm runs until a graph with higher chromatic number is found, or, in an unsuccessful attempt, all points of $A$ have been added to the graph.

We give our version of Mann's algorithm as it is applied to $G(\mathbb{Q}^3, \sqrt{t})$, and afterward, some commentary on the individual steps. 

\begin{enumerate}[label={\textbf{Step \arabic*:}},itemindent=1.5em,itemsep=5pt]

\item Select $V_s \subset \mathbb{Q}^3$ as an initial set of vertices, along with a set $A \subset \mathbb{Q}^3$.

\item Attempt to find a proper 3-coloring of $G$, the graph with induced vertex set $V_s$.  If no such 3-coloring can be found, stop and output $G$.  Otherwise, store the 3-coloring that is produced.

\item Compute all vertices of $A \setminus V_s$ that are adjacent to those of $V_s$.  Call this new set of vertices $V_h$.   

\item  Select a ``best" candidate $v_0 \in V_h$ to be a vertex in our graph.  Use the following criteria.
\begin{itemize}
		\item Among all $v \in V_h$, let $v_0$ be one which is adjacent to the largest number of vertices already in $V_s$.
		\item If multiple $v \in V_h$ are tied for having the most neighbors in $V_s$, select $v_0$ which has the highest number of distinct colors among those that it is adjacent to.
				\end{itemize}
\item Let $V_s' = V_s \cup \{v_0\}$ and restart Step 1 of the algorithm with initial set of vertices $V_s'$.
\end{enumerate}

We were able to successfully employ the above algorithm to find a 4-chromatic subgraph $H$ of $G(\mathbb{Q}^3, \sqrt{22})$.  The graph $H$ was of order 56, although we were ultimately able to find a 4-critical subgraph $H'$ of $H$, which was of order 29.  It is given in the Appendix.

It should be noted that some care must be taken in choosing $A$.  In Mann's description of the algorithm, he lets points of $A$ be those of $\mathbb{Q}^n$ having each of their coordinates of the form $\frac{x_i}{2^j}$ where $x_i$ is an integer and $j \in \{0,1\}$.  This is done in an effort to reduce computing time in Steps 2 and 3.  In $\mathbb{Q}^3$, however, several problems arise with this setup.  Consider a vector $v = \langle \frac{a}{d}, \frac{b}{d}, \frac{c}{d} \rangle$ where $\gcd(a,b,c,d) = 1$ and $|v| = \sqrt{t}$.  Then $a^2 + b^2 + c^2 = td^2$, and since $0$ and $1$ are the only quadratic residues of 4, $t \equiv 2 \pmod 4$ implies either zero or two of $a,b,c$ are odd.  It can't be the case that all of $a,b,c$ are even, as that would imply $td^2 \equiv 0 \pmod 4$ which results in $d$ even as well, and if two of $a,b,c$ are odd, then $a^2 + b^2 + c^2 \equiv 2 \pmod 4$ implies $d$ is odd.  It's natural to include at least some points of $\mathbb{Z}^3$ in the original $V_s$, but in doing so, the above observations show that $A$ should not contain $\mathbb{Q}^3$ points having any of their coordinates being reduced fractions with even denominators. 

If $t \equiv 1 \pmod 3$, the set $A$ should have at least some points with coordinates being reduced fractions whose denominators are divisible by $3$.  This is due to Theorem \ref{3coloringtheorem} below, and the simple observation that any subgraph $K$ of $G(\mathbb{Q}^3, \sqrt{t})$ also appears as a subgraph of $G(\mathbb{Z}^3, \sqrt{k^2t})$, where $k$ is the product of all denominators of coordinates entries of vertices of $K$.  If $t \equiv 1 \pmod 3$ and $3$ does not divide $k$, then $k^2t \equiv 1 \pmod 3$ as well.

\begin{theorem} \label{3coloringtheorem} Let $d \in \mathbb{Z}^+$ be odd with $d \equiv 2 \pmod 3$.  Then $\chi(\mathbb{Z}^3, \sqrt{2d}) = 3$.
\end{theorem}

\begin{proof} As shown in \cite{chow}, $\chi(\mathbb{Z}^3, \sqrt{2d}) \geq 3$.  Note that $2d \equiv 1 \pmod 3$, and since $x^2$ is congruent to 0 or 1 modulo 3 for all integers $x$, any vector $v \in \mathbb{Z}^3$ where $v = \langle x_0, y_0, z_0 \rangle$ and $|v| = \sqrt{2d}$ will have exactly two of $x_0, y_0, z_0$ congruent to 0 modulo 3.  It follows that, for any $p_1 = (x_1, y_1, z_1), p_2 = (x_2, y_2, z_2) \in \mathbb{Z}^3$ with $|p_1 - p_2| = \sqrt{2d}$, exactly two of the differences $(x_1 - x_2), (y_1 - y_2), (z_1 - z_2)$ will be congruent to 0 modulo 3.  The coloring $\varphi: \mathbb{Z}^3 \to \{0,1,2\}$ where $\varphi(x,y,z) \equiv x+y+z \pmod 3$ then guarantees that any adjacent vertices in $G(\mathbb{Z}^3, \sqrt{2d})$ receive different colors.\qed
\end{proof}

For Step 1, it certainly makes sense to begin the algorithm with an initial input graph $G$ having $\chi(G) = 3$, and indeed, in the implementation for $t = 22$, we let $V_s = \{(0,0,0), (\frac{14}{3},\frac{1}{3},\frac{1}{3}), (\frac{19}{3},\frac{-1}{3},\frac{14}{3}), (6,0,0), (3,3,2)\}$, which constitute the vertices of a $5$-cycle in $G(\mathbb{Q}^3, \sqrt{22})$.  Our set $A$ simply consisted of all $\mathbb{Q}^3$ points $p$ such that $3p \in \mathbb{Z}^3$.  Regarding Step 4, as noted in \cite{mann}, it is not at all guaranteed that starting with a $3$-colored graph $G$, and then placing a new vertex that is adjacent to vertices of all three colors will result in a graph $G'$ with $\chi(G') = 4$.  However, this does appear to be a good heuristic to aid in the selection of $v_0$.      
 
We also note that the selection of the initial set $V_s$ is crucial to the success (or failure) of the search algorithm,  not only with respect to the subgraph of $G(\mathbb{Q}^3, \sqrt{t})$ that is induced by $V_s$, but the points of $V_s$ themselves.  We placed a cap of 1000 on the number of vertices to add to the graph, where, if the cap was reached, we terminated the program.  Several failed run-throughs of the algorithm, each with initial sets being different 5-cycles in $G(\mathbb{Q}^3, \sqrt{t})$, reached the cap while still being able to find a proper 3-coloring of the resulting graph. 

The fact that $\chi(\mathbb{Q}^3, \sqrt{22}) = 4$ is of some interest by itself, however, we decided not to pursue this method for other $t \in T$.  With the ultimate goal in this whole line of inquiry being a complete resolution of Question \ref{mainquestion}, it seems unlikely that this style of search would ever lead to success.  That said, in the next two sections, we will take a different angle and instead search for \textit{specific} 4-chromatic subgraphs of $G(\mathbb{Q}, \sqrt{t})$.

\section{Finding a Gr\"{o}tzsch-type Graph}

It is well-known that the Gr\"{o}tzsch graph is the triangle-free 4-chromatic graph of minimum order.  For reference, it is given in Figure \ref{grotzsch} with a standard labeling of its vertices.

\begin{figure}[h]%
\begin{center}
\includegraphics[scale=1.1]{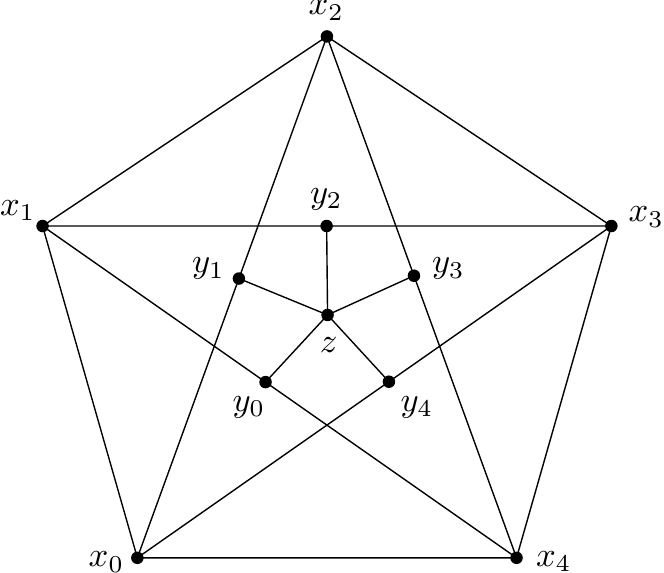}%
\caption{}%
\label{grotzsch}%
\end{center}
\end{figure}

The authors have tried to construct the Gr\"{o}tzsch graph as a subgraph of $G(\mathbb{Q}^3, \sqrt{t})$ for various $t$, but were ultimately unsuccessful.  The primary source of difficulty lies in somehow finding a way to place all five of the vertices $y_0, \ldots, y_4$ at rational points on a sphere having radius $\sqrt{t}$ and also being centered at a rational point.  The standard drawing of the Gr\"{o}tzsch graph in Figure \ref{grotzsch} hides this difficulty and is somewhat misleading, as any representation of the graph as a Euclidean distance graph in $\mathbb{Q}^3$ would not exhibit this amount of radial symmetry.  A classical argument gives the ratio of diagonal length to side length of a regular pentagon as $\frac{1 + \sqrt5}{2}$, and since any ratio of distances each realized between points of $\mathbb{Q}^3$ must be of the form $\sqrt{q}$ for some $q \in \mathbb{Q}$, one can see the impossibility of arranging either $x_0, \ldots, x_4$ or $y_0, \ldots, y_4$ as the vertices of a regular pentagon.

In this section, we will describe a method used to search for (and occasionally, to find) a representation in some $G(\mathbb{Q}^3, \sqrt{t})$ of a 4-chromatic graph $G$ which is similar to the Gr\"{o}tzsch graph.  We first give a rough version of the algorithm, and we will then follow with elaboration on the individual steps, along with proof that $\chi(G) = 4$.  Throughout, any subscript is computed modulo 5.

\begin{enumerate}[label={\textbf{Step \arabic*:}},itemindent=1.5em,itemsep=5pt]

\item Find a 5-cycle in $G(\mathbb{Q}^3, \sqrt{t})$.  Label the vertices of this 5-cycle as $v_0, \ldots, v_4$.

\item For each $i \in \{0, \ldots, 4\}$, let $C_i$ be the circle in $\mathbb{R}^3$ consisting of all points simultaneously at distance $\sqrt{t}$ from each of the vertices $v_{i-1}$ and $v_{i+1}$.

\item  For each $i \in \{0, \ldots, 4\}$, parameterize the points $C_i \cap \mathbb{Q}^3$ in terms of a single rational parameter $s_i$.

\item For each $i \in \{0, \ldots, 4\}$, define a collection of three circles $\mathscr{C}_i = \{C_{i-1}, C_i, C_{i+1}\}$. 

\item Form a suitably large collection of rational numbers, and let $\mathscr{L}$ be the set of all 3-tuples with entries from this collection. 

\item For each $i \in \{0, \ldots, 4\}$, perform the following procedure.
		\begin{itemize}
		\item Select a $3$-tuple $(x,y,z) \in \mathscr{L}$.  
		\item Select rational points $X_{i-1}, Y_i, Z_{i+1}$ on the circles $C_{i-1}, C_i, C_{i+1}$, respectively, by plugging $x = s_{i-1}, y = s_i, z = s_{i+1}$ into the parameterizations given in Step 3.
		\item If the circumradius of points $X_{i-1}, Y_i, Z_{i+1}$ is less than or equal to $\sqrt{t}$, find a point $Q_i$ simultaneously at distance $\sqrt{t}$ from each of $X_{i-1}, Y_i, Z_{i+1}$.  If $Q_i \in \mathbb{Q}^3$, stop.  Otherwise, repeat this step with a new 3-tuple from $\mathscr{L}$. 
		\end{itemize}
\end{enumerate}

If Step 6 has been successfully completed for each $\mathscr{C}_i$, our output is a graph $G$ of order 25 (the vertices $v_0, \ldots, v_4$ along with five sets of four vertices found in Step 6).  This graph is drawn in the Appendix as Figure \ref{grotzschtypegraph}, however to see that $\chi(G) = 4$, it is easier to consider the method in which $G$ was constructed.  We do this below.

\begin{theorem} The graph $G$ resulting from the above algorithm has $\chi(G) = 4$.
\end{theorem}

\begin{proof} To see that $G$ admits a proper 4-coloring, assume to the contrary that $\chi(G) > 4$, which means that none of the twenty vertices of $G$ having degree 3 are critical.  A new graph $G'$ can be formed by deleting each of these twenty vertices from $G$, and we have that $\chi(G') = \chi(G)$.  However, $G'$ is isomorphic to $C_5$ and is thus 3-colorable.

Consider a copy of the cycle $C_5$, which has been properly 3-colored, say with colors red, green, and blue.  Excluding symmetries and permutations of the colors, such a coloring is unique, and if the vertices of the 5-cycle have been labeled $v_0, \ldots, v_4$ in the usual fashion, observe that for some $i \in \{0, \ldots, 4\}$, the set $\{v_{i-1}, v_i, v_{i+1}\}$ will consist of a red vertex that is adjacent to blue and green vertices, a blue vertex that is adjacent to red and green vertices, and a green vertex that is adjacent to red and blue vertices.  For that same $i$, it follows that in any proper 3-coloring of $G$, the vertices $X_{i-1}, Y_{i}, Z_{i+1}$ will each receive different colors, and in turn, a fourth color is required for vertex $Q_i$.\qed   
\end{proof}

Regarding Step 1, we have been unable to show that the cycle $C_5$ is a subgraph of $G(\mathbb{Q}^3, \sqrt{t})$ for all $t \in T$, nor can we find any result in the literature indicating so.  The existence of a 5-cycle in each $G(\mathbb{Q}^3, \sqrt{t})$ is asked in \cite{bau}, and we will pose the question here as well.

\begin{question} \label{5cyclequestion} For each $t \in T$, does the graph $G(\mathbb{Q}^3, \sqrt{t})$ have $C_5$ as a subgraph?
\end{question}

In \cite{gaston}, an \textit{odd vector cycle} is defined as a collection of an odd number of $\mathbb{Z}^m$ vectors, each of a given magnitude $\sqrt{r}$, whose sum equals the zero vector.  Given positive integer $r$ for which an odd vector cycle actually exists, the function $C_m(r)$ is defined in \cite{gaston} as the minimum possible number of vectors in the collection, and $C_3(r)$ is determined via computer search for all $r < 10^6$.  Due to the results of this search, we have that $G(\mathbb{Q}^3, \sqrt{t})$ has $C_5$ as a subgraph for all $t \in T$ satisfying $t < 10^6$.  This leads us to strongly believe that Question \ref{5cyclequestion} has a positive answer.  In practice, 5-cycles have been quite easy to find in $G(\mathbb{Q}^3, \sqrt{t})$, and to execute Step 1, one may generate a large number of $\mathbb{Q}^3$ vectors, each of length $\sqrt{t}$, and then use some type of ``meet in the middle" algorithm to hopefully find five of them that sum to the zero vector.  Note also that the vectors should be selected so that no points $v_{i-1}, v_{i}, v_{i+1}$ are collinear, as that would mean that the circle $C_i$ is degenerate, consisting only of the point $v_{i}$ itself.

Step 4 is straightforward to execute in light of the following Theorem \ref{nagellparameter}, which is seen with proof in \cite{nagell}.

\begin{theorem} \label{nagellparameter} Let $ax^2 + bxy + cy^2 + dx + ey + f = 0$ be the equation of a conic where $a, b, c, d, e, f \in \mathbb{Q}$.  Suppose $(\xi, \eta)$ is a rational point on the conic.  Additional rational points $(x,y)$ on the conic have the following parameterization where $s$ runs through all the rational numbers.\\
\begin{center}
$\ds x = \frac{-d - a\xi -b\eta -(2c\eta + e)s + c\xi s^2}{a + bs + cs^2}$, \, $\ds y = \frac{a\eta - (2a\xi + d)s - (b\xi + c\eta + e)s^2}{a + bs + cs^2}$\\
\end{center}
\noindent The only rational point not obtained through this parameterization (should it actually exist on the conic) is the point $(\xi, \frac{-b\xi - c\eta - e}{c})$ and is found by letting $s$ approach $\infty$.
\end{theorem}
\vspace{15pt}

The application of Theorem \ref{nagellparameter} to parameterize the rational points on each circle $C_i$ is as follows.  For each $i \in \{0, \ldots, 4\}$, let $\mathcal{P}_i$ be the plane containing $C_i$, and designate $M_i$ to be the midpoint of points $v_{i-1}$ and $v_{i+1}$.  Since $\mathcal{P}_i$ has normal vector $v_{i+1} - v_{i-1}$ and contains $M_i$, we have that $\mathcal{P}_i$ is given by an equation of the form $\alpha_ix + \beta_iy + \gamma_iz = w_i$ for some $\alpha_i,\beta_i,\gamma_i,w_i \in \mathbb{Q}$.  Letting $v_{i-1} = (x_{i-1}, y_{i-1}, z_{i-1})$, since any point on $C_i$ is at distance $\sqrt{t}$ from $v_{i-1}$, we have that $(x,y,z) \in C_i$ satisfies $(x - x_{i-1})^2 + (y - y_{i-1})^2 + (z - z_{i-1})^2 = t$.  At least one of $\alpha_i,\beta_i,\gamma_i$ is non-zero, so without loss of generality assume that $\gamma_i \neq 0$ and substitute to rewrite the preceding equation as $(x - x_{i-1})^2 + (y - y_{i-1})^2 + (\frac{\alpha_ix + \beta_iy}{-\gamma_i} - z_{i-1})^2 = t$ which is an equation of the form stipulated by Theorem \ref{nagellparameter}.  Furthermore, since $v_i$ is a point on circle $C_i$, this equation has solution $x = x_i, y = y_i$ where $v_i = (x_i, y_i, z_i)$.  Theorem \ref{nagellparameter} then allows us to parameterize the $(x,y)$ solutions to this equation in terms of a single rational parameter $s_i$, and we may then find an expression in $s_i$ for the $z$-coordinate of a point $(x,y,z) \in C_i$ by substituting those parameterizations for $x$ and $y$ back into the equation $\alpha_ix + \beta_iy + \gamma_iz = w_i$.       

Using this algorithm, we successfully found a $4$-chromatic subgraph of $G(\mathbb{Q}^3, \sqrt{t})$ for a few small values of $t$, namely, $t \in \{34,66\}$.  The vertex sets of these subgraphs are given in the Appendix.  Even though the number of successes is admittedly few, we feel that this is a promising strategy for an eventual resolution of Question \ref{mainquestion}, for the following reasons.

Consider for a moment three circles in $\mathbb{R}^3$, notated as $C_{\alpha}$, $C_{\beta}$, and $C_{\gamma}$, and suppose that rational points on these circles have been parameterized using Theorem \ref{nagellparameter} above, say, using rational parameters $s_{\alpha}$, $s_{\beta}$, and $s_{\gamma}$, respectively.  Given specific selections of $s_{\alpha}$, $s_{\beta}$, $s_{\gamma}$ which produce points $P_{\alpha}$, $P_{\beta}$, and $P_{\gamma}$ on the corresponding circles, the circumcenter $P_0$ of the triangle with vertices $P_{\alpha}$, $P_{\beta}$, $P_{\gamma}$ is guaranteed to be a point of $\mathbb{Q}^3$.  As well, there exists a vector $n = \langle q_1, q_2, q_3 \rangle$ normal to the plane containing $P_{\alpha}$, $P_{\beta}$, $P_{\gamma}$ where each of $q_1, q_2, q_3$ are rational.  Let $\ell$ be the set of all $\mathbb{R}^3$ points equidistant from $P_{\alpha}$, $P_{\beta}$, and $P_{\gamma}$, or, in other words, the line which is parallel to vector $n$ and passes through the point $P_0$.  There are infinitely many rational points on $\ell$ and they are of the form $(q_1s + x_0, q_2s + y_0, q_3s + z_0)$ where $s \in \mathbb{Q}$ and $P_0 = (x_0,y_0,z_0)$.  If the circumradius $r$ of the triangle with vertices $P_{\alpha}$, $P_{\beta}$, $P_{\gamma}$ satisfies $r \leq \sqrt{t}$, there is a point $Q$ on $\ell$ at distance $\sqrt{t}$ from $P_{\alpha}$, $P_{\beta}$, $P_{\gamma}$.  Applying the Pythagorean Theorem and consulting Figure \ref{righttriangle} below, one can see that $Q$ is given when $r^2 + (q_1s)^2 + (q_2s)^2 + (q_3s)^2 = t$.  For $Q$ to be a point of  $\mathbb{Q}^3$, all we need is for $s^2 = \frac{t - r^2}{{q_1}^2 + {q_2}^2 + {q_3}^2}$ to be a perfect rational square.  This is of course a Diophantine equation, and it is our hope that it can be attacked via techniques from classical number theory.  Unfortunately, the expressions for $q_1$, $q_2$, $q_3$, and $r$ in terms of rational parameters $s_{\alpha}$, $s_{\beta}$, $s_{\gamma}$ are extremely unwieldy, and we see no method of guaranteeing a selection of $s_{\alpha}$, $s_{\beta}$, $s_{\gamma}$ that will result in the equation being solvable.  As such, we were forced to fall back on Step 5 as given in the algorithm, creating a list of rational inputs, and then just hoping to get lucky.  

\begin{figure}[h]%
\begin{center}
\includegraphics[scale=1.1]{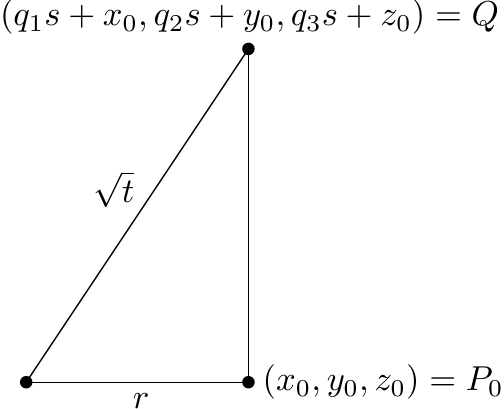}%
\caption{}%
\label{righttriangle}%
\end{center}
\end{figure}

\section{Finding a Gr\"{o}tzsch Subgraph}

In this section, we will describe another search method that we have employed to show that $\chi(\mathbb{Q}^3, \sqrt{t}) = 4$ for some $t \in T$.  It is a more formal extension of that used in \cite{noble1} to show that $\chi(\mathbb{Q}^3, \sqrt{10}) = 4$.  Consider the graph $H$ in Figure \ref{grotzschsubgraph} and note that $H$ is a subgraph of the Gr\"{o}tzsch graph, formed by deleting the vertex labeled $y_2$ in Figure \ref{grotzsch}.

\begin{figure}[h]%
\begin{center}
\includegraphics[scale=.6]{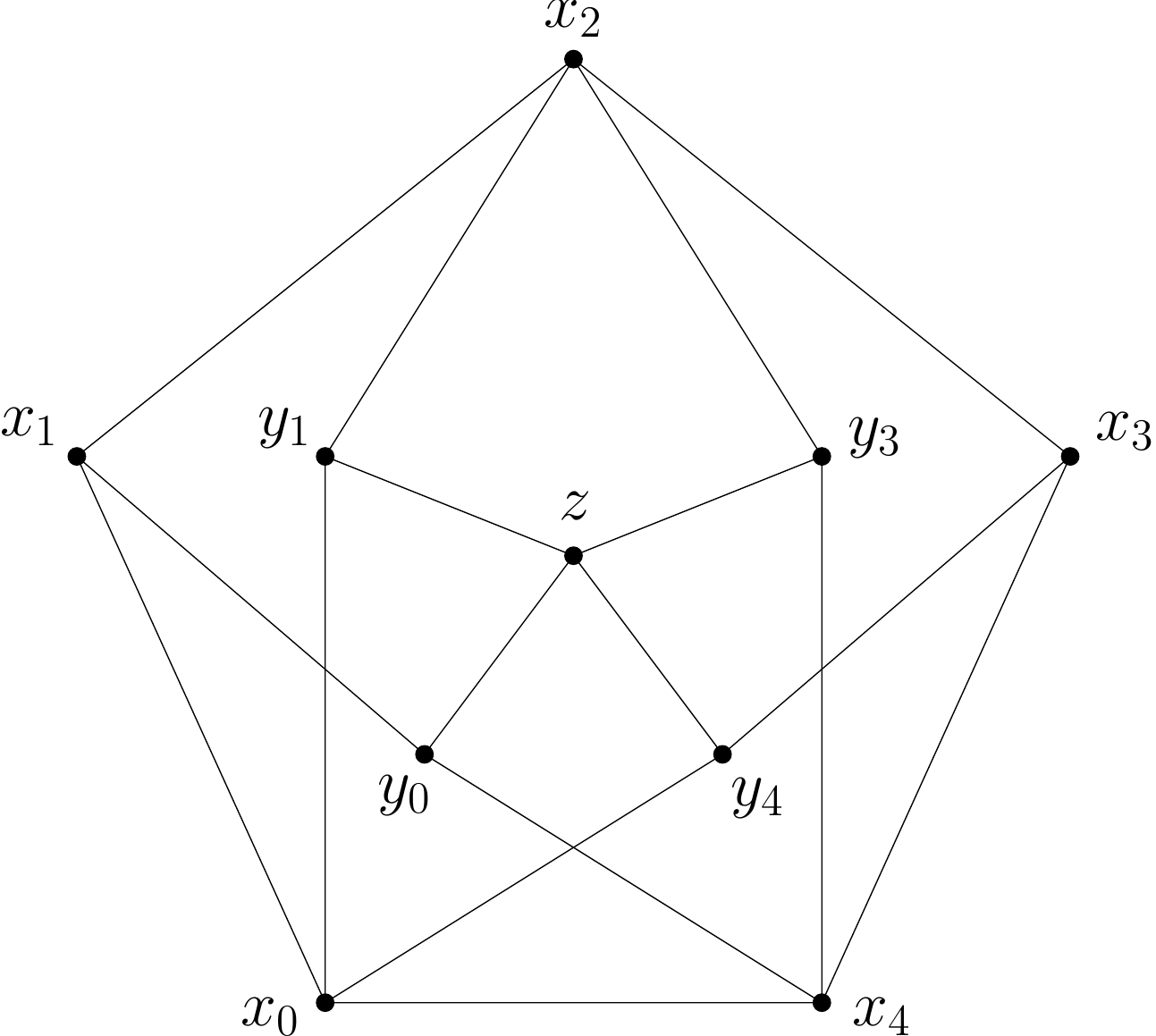}%
\caption{}%
\label{grotzschsubgraph}%
\end{center}
\end{figure}

Since the Gr\"{o}tzsch graph $G$ is 4-critical and $H = G \setminus \{y_2\}$, we have that $\chi(H) = 3$.  However, in any proper 3-coloring of $H$, the vertices labeled $x_2$ and $z$ must receive the same color.  This is fairly easy to see, as the criticality of $y_2$ in $G$ implies that $x_1, x_3, z$ must receive different colors in any proper 3-coloring of $H$, and since $x_2$ is adjacent to both $x_1$ and $x_3$, it must be that $x_2$ and $z$ are colored the same.  This observation was utilized in \cite{noble1} to eventually show that $\chi(\mathbb{Q}^3, \sqrt{10}) = 4$.  Here, however, we will also take advantage of the fact that $x_1, x_3$ must be colored differently. 

The following two theorems are central to the success of our search algorithm, and for a proof of Theorem \ref{rationalrotation}, see \cite{bau}.  The overall goal will be to construct a copy of $H$ appearing as a subgraph of the desired $G(\mathbb{Q}^3, \sqrt{t})$, and then assume to the contrary that $\chi(\mathbb{Q}^3, \sqrt{t}) = 3$.  Theorem \ref{rationalrotation} implies that not only must the points $x_2$ and $z$ receive the same color in a proper 3-coloring of $G(\mathbb{Q}^3, \sqrt{t})$, any $\mathbb{Q}^3$ vector of length $|x_2 - z|$ must have its initial and terminal points colored the same color.  As well, since $x_1$ and $x_3$ must be colored differently, the circle $S$ consisting of all points simultaneously at distance $\sqrt{t}$ from each of $x_1$ and $x_3$ must have all of its rational points colored the same color.  If $v$ is a vector with initial and terminal point being rational points $S$, Theorem \ref{rationalrotation} gives that any $\mathbb{Q}^3$ vector of length $|v|$ must have its initial and terminal points colored the same color.  If $|v|$ or $|x_2 - z|$ happen to fit the criteria given in Theorem \ref{vectorlength}, then in this supposed proper 3-coloring of $\mathbb{Q}^3$, there exists a finite sequence of points, each colored the same color, with the first and last points of the sequence being distance $\sqrt{t}$ apart.

\begin{theorem} \label{rationalrotation} Let $n \geq 1$ and $v_1, v_2 \in \mathbb{Q}^n$ with $|v_1| = |v_2|$.  There exists an isometry $\varphi: \mathbb{Q}^n \to \mathbb{Q}^n$ such that $\varphi(v_1) = v_2$.
\end{theorem}

\begin{theorem} \label{vectorlength} Let $v \in \mathbb{Q}^3$ with $|v| = \sqrt{t}$ for some $t \in T$, and suppose $h = \frac{m}{n}$ for $m,n \in \mathbb{Z}^+$ with $\gcd(m,n) = 1$.  Let $S$ be the set of all $\mathbb{Q}^3$ vectors of length $\sqrt{h}$, and let $\Phi(S)$ be the set of all vectors generated by those of $S$ under the usual vector addition.  If any of the following hold, then $v \in \Phi(S)$.
\begin{itemize}
\item $m \equiv 2 \pmod 4$.
\item The square-free part $n_0$ of $n$ is even.
\item The square-free part $n_0$ of $n$ is odd and $mn_0 \equiv 1 \pmod 4$.
\end{itemize}
\end{theorem}

\begin{proof} Throughout, we will use the elementary fact that $x^2$ is congruent to $0$ or $1$ modulo $4$ for all integers $x$.  Let $v = (x_0,y_0,z_0)$ and since $t \equiv 2 \pmod 4$, we have that exactly two of $x_0,y_0,z_0$ are odd.  

First, assume that $m \equiv 2 \pmod 4$.  By a classical result of Legendre, there exist $a,b,c \in \mathbb{Z}$ such that $a^2 + b^2 + c^2 = mn$ and $\gcd(a,b,c) = 1$.  Note that $mn \equiv 2 \pmod 4$ implies that exactly two of $a,b,c$ are odd.  We have $|(\frac{a}{n}, \frac{b}{n}, \frac{c}{n})| = \sqrt{h}$, and with $\Phi(S)$ closed under addition, $(a,b,c) \in \Phi(S)$.  Note that $\Phi(S)$ is closed under the operations of permuting coordinate entries of a vector or replacing any of those entries with their corresponding negatives, and as such, $(a,b,c) + (a,-b,-c) = (2a,0,0) \in \Phi(S)$, and similarly,  $(2b,0,0), (2c,0,0) \in \Phi(S)$ as well.  Since $\gcd(a,b,c) = 1$, there exist integers $d_1, d_2, d_3$ such that $d_1(2a,0,0) + d_2(2b,0,0) + d_3(2c,0,0) = (2,0,0)$.  Having now established that $(\pm 2,0,0), (0, \pm 2, 0), (0,0, \pm 2) \in \Phi(S)$, we may start with $(a,b,c)$ and repeatedly add to it some number of copies of the vectors $(\pm 2,0,0)$, $(0, \pm 2, 0)$, $(0,0, \pm 2)$ to construct $v$.

Now assume instead that the square-free part $n_0$ of $n$ is even.  Write $n = 4^\alpha n_0$ where $n_0 \equiv 2 \pmod 4$ and let $a^2 + b^2 + c^2 = 4^\alpha n_0m$.  This implies that $2^\alpha$ divides each of $a,b,c$, so write $a = 2^\alpha a_0$, $b = 2^\alpha b_0$, $c = 2^\alpha c_0$, and after cancellation, we have ${a_0}^2 + {b_0}^2 + {c_0}^2 = n_0m$.  By Legendre's result, $a_0, b_0, c_0$ can be assumed to be relatively prime, and the argument precedes just as in the previous case.

Finally, assume that the square-free part $n_0$ of $n$ is odd and that $mn_0 \equiv 1 \pmod 4$.  Just as above, we may write $mn = 4^\alpha n_0m$ and then let ${a_0}^2 + {b_0}^2 + {c_0}^2 = mn_0$ where $\gcd(a_0, b_0, c_0) = 1$.  As $mn_0 \equiv 1 \pmod 4$, exactly one of $a_0, b_0, c_0$ is odd.  We may repeat the argument from the first case to obtain the fact that $(\pm 2,0,0), (0, \pm 2, 0), (0,0, \pm 2) \in \Phi(S)$.  Without loss of generality, assume that $a_0$ is odd, and repeatedly add some number of the vectors $(\pm 2,0,0), (0, \pm 2, 0), (0,0, \pm 2)$ to $(a_0, b_0, c_0)$ to obtain $(1,0,0) \in \Phi(S)$.  It immediately follows that $v \in \Phi(S)$.\qed
\end{proof}

Our search algorithm is as follows. 

\begin{enumerate}[label={\textbf{Step \arabic*:}},itemindent=1.5em,itemsep=5pt]

\item Find a symmetric 5-cycle in $G(\mathbb{Q}^3, \sqrt{t})$.  By this, we mean vertices $x_0, \ldots, x_4 \in \mathbb{Q}^3$ such that $|x_i - x_{i+1}| = \sqrt{t}$ for $i \in \{0, \ldots, 4\}$, and, letting $\mathcal{P}$ designate the plane consisting of all points equidistant from $x_0$ and $x_4$, we have $x_2$ and the midpoint of $x_1$, $x_3$ both lying in $\mathcal{P}$.

\item For $i \in \{0, 1, 3, 4\}$, let $C_i$ be the circle in $\mathbb{R}^3$ consisting of all points simultaneously at distance $\sqrt{t}$ from each of the vertices $x_{i-1}$ and $x_{i+1}$.

\item  For $i \in \{0, 1\}$, parameterize the points $C_i \cap \mathbb{Q}^3$ in terms of a single rational parameter $t_i$.

\item Form a suitably large collection of rational numbers, and let $\mathscr{L}$ be the set of all ordered pairs with entries from this collection. 

\item Perform the following procedure.
		\begin{itemize}
		\item Select $(a,b) \in \mathscr{L}$.  
		\item Select rational points $y_0 \in C_0$, $y_1 \in C_1$ by plugging $a = t_0, b = t_1$ into the corresponding parameterizations given in Step 3.
		\item Let $C$ be the set of all points simultaneously at distance $\sqrt{t}$ from $y_0, y_1$, and determine a point of intersection (should it exist) of $C$ and $\mathcal{P}$.  If this point of intersection is in $\mathbb{Q}^3$, call it $z$ and stop.  Otherwise, repeat this step with a new ordered pair from $\mathscr{L}$.  
		\end{itemize}
		
\item Determine the distance $|x_2 - z|$.  If it meets the criteria given in Theorem \ref{vectorlength}, stop.  If $|x_2 - z|$ does not fit the criteria given in Theorem \ref{vectorlength}, denote by $S$ the circle consisting of all points simultaneously at distance $\sqrt{t}$ from each of $x_1$ and $x_3$, and let $\sqrt{q}$ be the radius of $S$ for some $q \in \mathbb{Q}$.  If $q$ is of the form $q = \frac{m}{n}$ where $\gcd(m,n) = 1$ and $n \equiv 2 \pmod 4$, stop.  Otherwise, repeat the previous step.
\end{enumerate} 

Some elaboration is needed.  The reason we require $x_0, \ldots, x_4$ to be the vertices of a symmetric 5-cycle is that, once $z \in \mathcal{P}$ has been found, we are guaranteed that there exist rational points $y_3 \in C_3, y_4 \in C_4$ with $|y_3 - z| = |y_4 - z| = \sqrt{t}$.  Regarding Step 6, if the radius of $S$ is indeed of the form $\sqrt{\frac{m}{n}}$ with $n = 2p$ for some odd integer $p$, then antipodal points of $S$ are distance $\sqrt{\frac{2m}{p}}$ apart.  We have $\sqrt{\frac{2m}{p}}$ adhering to the conditions given in Theorem \ref{vectorlength}, and it follows that $\chi(\mathbb{Q}^3, \sqrt{t}) = 4$.  It should also be said that, unfortunately, if the radius of $S$ is not of this form, there do not exist rational points $p_1, p_2 \in S$ where $|p_1 - p_2|$ fits the criteria of Theorem \ref{vectorlength}.  We omit proof of this fact. 

As one might expect, especially in light of Question \ref{5cyclequestion} from the previous section, we have been unable to supply proof that a symmetric 5-cycle exists in $G(\mathbb{Q}^3, \sqrt{t})$ for all $t \in T$.  We will pose that as Question \ref{symmetric5cyclequestion} below, where, of course, a positive answer to Question \ref{symmetric5cyclequestion} would imply a positive answer to Question \ref{5cyclequestion} as well.

\begin{question} \label{symmetric5cyclequestion} For each $t \in T$, does the graph $G(\mathbb{Q}^3, \sqrt{t})$ have a symmetric 5-cycle?
\end{question}

In \cite{noble2}, the second author obtains a characterization of the possible side lengths of isosceles triangles which may be oriented so that their vertices are points of $\mathbb{Q}^3$.  It is given as Theorem \ref{isosceles}, and it offers a possible avenue of attack for answering Question \ref{symmetric5cyclequestion}.

\begin{theorem} \label{isosceles} Let $r,d \in \mathbb{Q}^+$ where $\sqrt{r}$ and $\sqrt{d}$ are both realized as distances in $\mathbb{Q}^3$.  Let $r = a^2 + b^2 + c^2$ for $a,b,c \in \mathbb{Q}$ with $a,b$ not both equal to zero.  Then the triangle $T(\sqrt{r}, \sqrt{d}, \sqrt{d})$ is embeddable in $\mathbb{Q}^3$ if and only if the Diophantine equation $x^2 + ry^2 - (4d - r)(a^2 + b^2)z^2 = 0$ has a non-trivial integer solution.
\end{theorem} 

In the original appearance of Theorem \ref{isosceles} in \cite{noble2}, $r, d$ were stipulated to be integers, but the above presentation holds as well.  Also, note that if one were given points $p_1, p_2 \in \mathbb{Q}^3$ with $|p_1 - p_2| = \sqrt{r}$, and were asked if there exists $p_3 \in \mathbb{Q}^3$ at distance $\sqrt{d}$ from each of $p_1, p_2$, it is not important where the $p_1, p_2$ are actually located.  This is a direct consequence of Theorem \ref{rationalrotation}, and all one needs to do to decide whether or not such a $p_3$ exists is to determine the solubility of the Diophantine equation in Theorem \ref{isosceles}.  With this in mind, consider Figure \ref{rd5cycle}.  The drawing is perhaps misleading as it depicts vertices $x_1$ and $x_3$ lying in the same plane as $x_0,x_2,x_4$, however, the exact placement of $x_1, x_3$ is not important.  For our means, all that really matters are the distances listed in Figure \ref{rd5cycle}.

\begin{figure}[h]%
\begin{center}
\includegraphics[scale=.5]{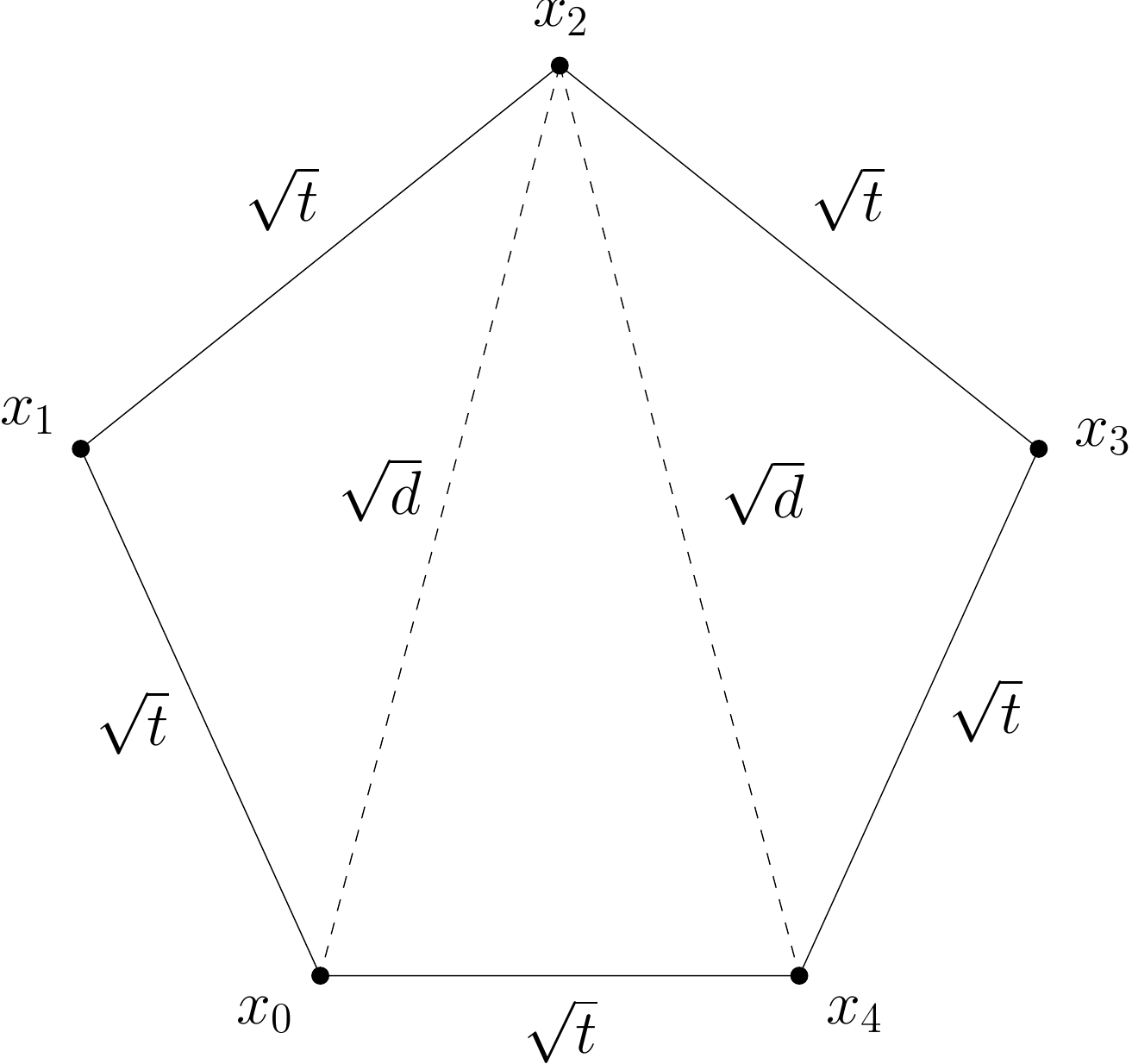}%
\caption{}%
\label{rd5cycle}%
\end{center}
\end{figure}

An arrangement of points $x_0, \ldots, x_4 \in \mathbb{Q}^3$ with distances between points consistent with Figure \ref{rd5cycle} is possible if and only if isosceles triangles with side lengths $\sqrt{t}, \sqrt{d}, \sqrt{d}$ and side lengths $\sqrt{d}, \sqrt{t}, \sqrt{t}$ are both embeddable in $\mathbb{Q}^3$.  Applying Theorem \ref{isosceles}, we have that both are embeddable in $\mathbb{Q}^3$ if and only if the following two Diophantine equations are solvable.  Here, $a_t, b_t, c_t$ are rationals such that ${a_t}^2 + {b_t}^2 + {c_t}^2 = t$ and $a_d, b_d, c_d$ are rationals such that ${a_d}^2 + {b_d}^2 + {c_d}^2 = d$.  Note also here that we are not requiring Equations \ref{eq1} and \ref{eq2} to have the \textit{same} solution $(x,y,z)$.  Rather, we just want each to have some non-trivial solution. 

\begin{equation}
x^2 + ty^2 - (4d - t)({a_t}^2 + {b_t}^2)z^2 = 0
\label{eq1}
\end{equation}

\begin{equation}
x^2 + dy^2 - (4t - d)({a_d}^2 + {b_d}^2)z^2 = 0
\label{eq2}
\end{equation}

As has been a frequent refrain in this article, we have been unable to determine if for every $t \in T$, there exists some corresponding $d \in \mathbb{Q}^+$ so that both Equation \ref{eq1} and Equation \ref{eq2} have a non-trivial solution in integers.  Experimental evidence suggests that there is such a $d$.  The authors constructed another algorithm employing Theorem \ref{legendretheorem}, which is also attributed to Legendre, and found that for all $t \in T$ with $t < 100,000$, there did in fact exist a corresponding $d \in \mathbb{Q}^+$ which resulted in Equations \ref{eq1} and \ref{eq2} simultaneously having a non-trivial solution.  More so, except for a few small values of $t$, we were able to find an integer for the desired $d$. 

\begin{theorem} \label{legendretheorem} Let $a, b, c$ be non-zero integers, not each positive or each negative, and suppose that $abc$ is square-free.  Then the equation
\begin{center} $ax^2 + by^2 + cz^2 = 0$
\end{center}
has a non-trivial integer solution $(x,y,z)$ if and only if each of the following are satisfied:

\begin{itemize} \item[$(i)$] $-ab$ is a quadratic residue of $c$
\item[$(ii)$] $-ac$ is a quadratic residue of $b$
\item[$(iii)$] $-bc$ is a quadratic residue of $a$.
\end{itemize}
\end{theorem}

Using the algorithm outlined in this section, we ultimately found the existence of a 4-chromatic subgraph of $G(\mathbb{Q}^3, \sqrt{30})$.  It is given in the Appendix.

\section{Concluding Thoughts}

It seems likely to the authors that, perhaps with some refinement, the methods presented in this paper (particularly those of Sections 3 and 4) could be implemented to search out $4$-chromatic subgraphs of $G(\mathbb{Q}^3, \sqrt{t})$ for many additional $t \in T$.  We have confidence that, should Question \ref{mainquestion} be eventually resolved, it will be answered in the negative.  We offer this as a conjecture below.\\

 \noindent\textbf{Conjecture} \,\,For any non-trivial graph $G(\mathbb{Q}^3, d)$, the chromatic number $\chi(\mathbb{Q}^3, d)$ is equal to 2 or 4.


\pagebreak

{\Large

\noindent \textbf{Appendix}\\

}

Figure \ref{28vertexgraph} depicts a $4$-chromatic subgraph of $G(\mathbb{Q}^3, \sqrt{22})$ found using the algorithm presented in Section 2.  Its chromatic number was obtained through use of a standard graph coloring program in Sage.  Note that in the figure below, the graph is not drawn as a Euclidean distance graph with all edges of equal length, as such a representation was quite difficult to visually comprehend, and we felt it did not add anything relevant to the discussion.

\begin{figure}[h]%
\begin{center}
\includegraphics[scale=.7]{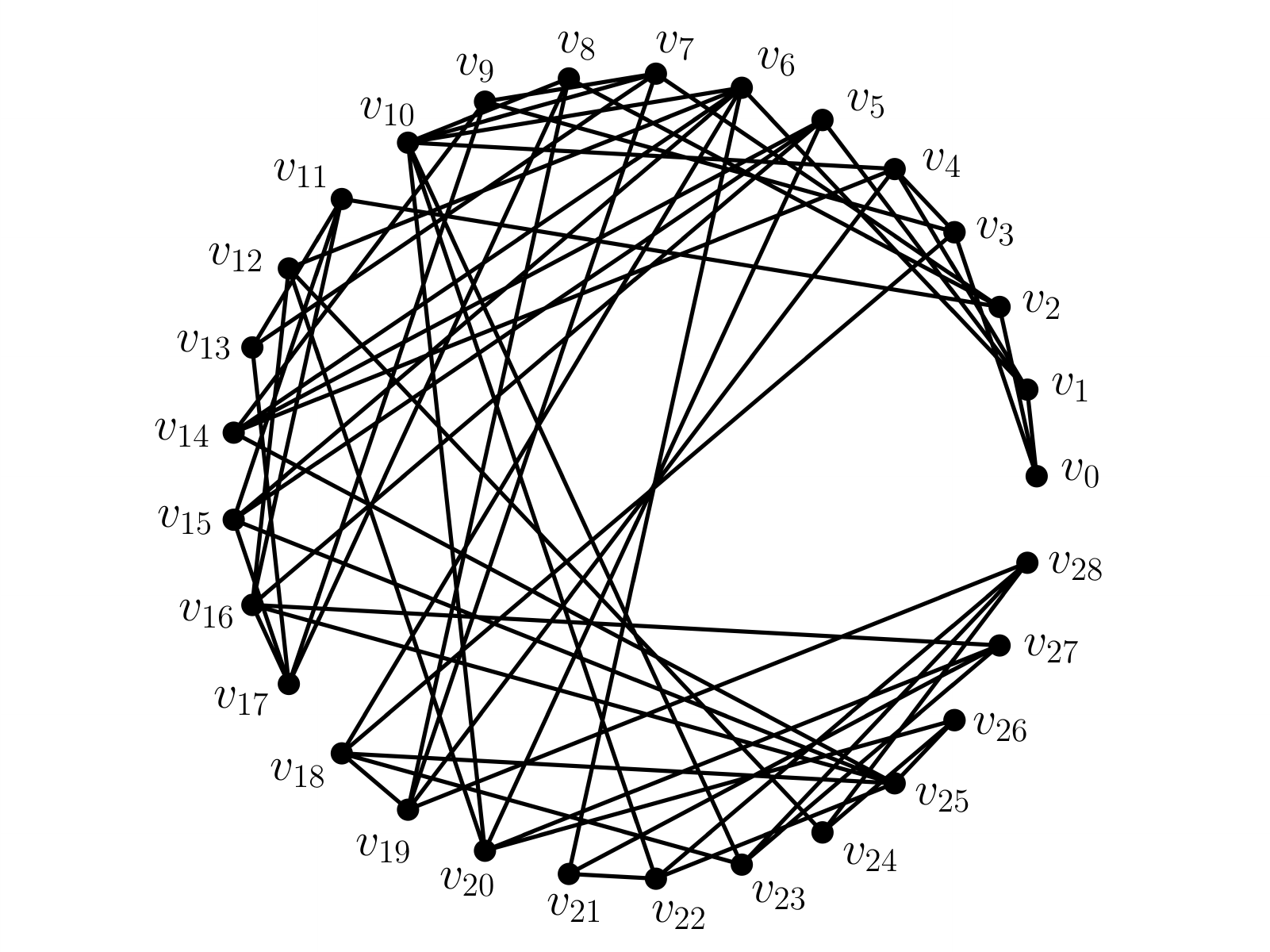}%
\caption{}%
\label{28vertexgraph}%
\end{center}
\end{figure}

\begin{center}

\renewcommand{\arraystretch}{1.5}

\resizebox{5.3in}{!} {
\begin{tabular}{|l|l|l|l|}

\hline

\multicolumn{4}{|c|}{Vertex Set of a 4-chromatic Subgraph of $G(\mathbb{Q}^3, \sqrt{22})$} \\ \hline

$v_0 = (-4, \frac{5}{3}, -1)$ 				&   $v_8 = (1,\frac53,4)$    &  $v_{16} = (3, 3, 2)$    &  $v_{24} = (\frac{17}{3}, \frac{4}{3}, \frac{1}{3})$    \\ \hline

$v_1 = (-\frac{7}{3}, \frac{7}{3}, \frac{10}{3})$					&   $v_9 = (\frac43,\frac43,-\frac43)$     &  $v_{17} = (\frac{10}{3}, -\frac{5}{3}, \frac{5}{3})$    &  $v_{25} = (6, 0, 0)$   \\ \hline

$v_2 = (-2,-\frac43,2)$					&   $v_{10} = (\frac{8}{3}, -\frac{8}{3}, \frac{10}{3})$     &  $v_{18} = (\frac{11}{3}, \frac{10}{3}, \frac{7}{3})$    &  $v_{26} = (\frac{19}{3}, -\frac{1}{3}, \frac{14}{3})$    \\ \hline

$v_3 = (-1, \frac{11}{3}, 2)$ 				&   $v_{11} = (\frac{8}{3}, -\frac{5}{3}, \frac{5}{3})$    &  $v_{19} = (4, -\frac{4}{3}, 2)$    &  $v_{27} = (\frac{23}{3}, \frac{10}{3}, \frac{7}{3})$    \\ \hline

$v_4 = (-\frac13, -\frac{2}{3}, \frac{1}{3})$					&   $v_{12} = (\frac{8}{3}, \frac{10}{3}, -\frac{8}{3})$    &  $v_{20} = (\frac{14}{3}, \frac{1}{3}, \frac{1}{3})$    &  $v_{28} = (\frac{26}{3}, -\frac{5}{3}, \frac{7}{3})$    \\ \hline

$v_5 = (0,0,0)$					&   $v_{13} = (3, -\frac{19}{3}, 2)$    &  $v_{21} = (\frac{16}{3}, 0, 0)$    &     \\ \hline

$v_6 = (\frac23, \frac13, \frac13)$					&   $v_{14} = (3, -3, -2)$    &  $v_{22} = (\frac{17}{3}, -\frac{14}{3}, \frac{1}{3})$    &     \\ \hline

$v_7 = (1,-\frac{10}{3},-1)$					&   $v_{15} = (3, -2, -3)$   &  $v_{23} = (\frac{17}{3}, \frac{1}{3}, \frac{16}{3})$    &     \\ \hline

\end{tabular}
}
\end{center}

\pagebreak

Figure \ref{grotzschtypegraph} contains the $4$-chromatic graph referenced in Section 3.  The charts that follow list the vertices of this graph as it was found as a subgraph of $G(\mathbb{Q}^3, \sqrt{t})$ for $t \in \{34, 66\}$.

\begin{figure}[h]%
\begin{center}
\includegraphics[scale=.4]{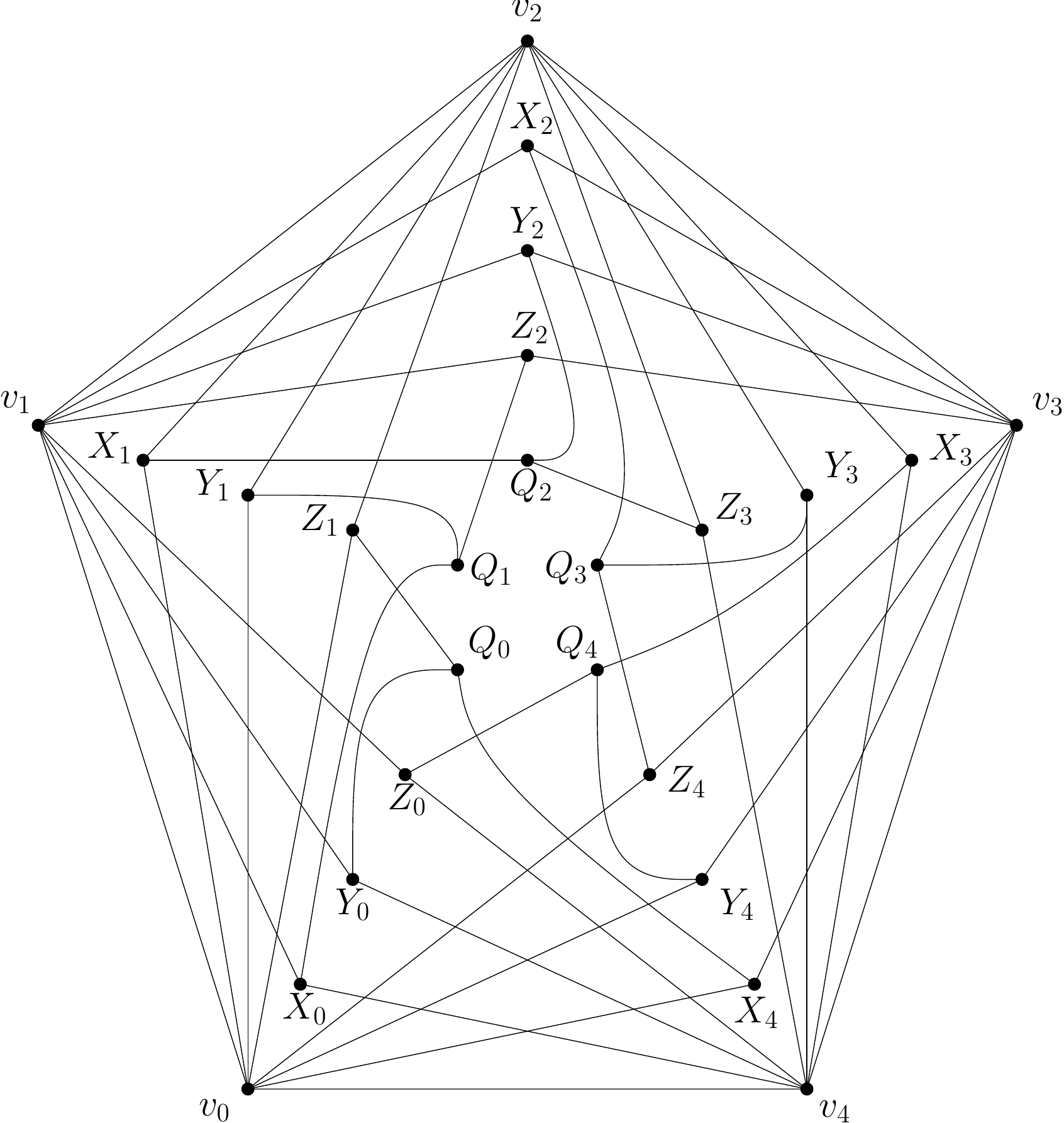}%
\caption{}%
\label{grotzschtypegraph}%
\end{center}
\end{figure}

\begin{center}

\renewcommand{\arraystretch}{1.9}

\resizebox{5.8in}{!} {
\begin{tabular}{|l|l|l|l|l|}

\hline

\multicolumn{5}{|c|}{Vertex Set of a 4-chromatic Subgraph of $G(\mathbb{Q}^3, \sqrt{34})$} \\ \hline

$v_0 = (0, 0, 0)$ 				&   $X_0 = (-\frac{36}{7}, -\frac{12}{7}, \frac{60}{7})$    &  $Y_{0} = (-\frac{9}{13}, -\frac{3}{13}, -\frac{12}{13})$    &  $Z_{0} = (-\frac{108}{11}, -\frac{36}{11}, \frac{36}{11})$ &  $Q_{0} = (-\frac{159}{227},-\frac{106}{227},\frac{1113}{227})$   \\ \hline

$v_1 = (-5, 0, 3)$ 				&   $X_1 = (-\frac{39}{7},-\frac17,\frac{12}{7})$    &  $Y_{1} = (-3, 5, 0)$    &  $Z_{1} = (-\frac{39}{7}, -\frac{1}{7}, \frac{12}{7})$ &  $Q_{1} =  (-\frac{2613}{803},\frac{236}{803},\frac{2757}{803})$  \\ \hline

$v_2 = (-8, 5, 3)$ 				&   $X_2 = (-\frac{16}{3},\frac{17}{3},\frac{13}{3})$    &  $Y_{2} = (-\frac{11}{3}, -\frac{11}{3}, -\frac{4}{3})$    &  $Z_{2} = (-\frac{16}{3}, \frac{17}{3}, \frac{13}{3})$ &  $Q_{2} =  (-\frac{429}{61},-\frac{344}{61},3)$  \\ \hline

$v_3 = (-4, 2, 0)$ 				&   $X_3 = (-\frac{64}{7}, -\frac{4}{7}, \frac{30}{7})$    &  $Y_{3} = (-\frac{8}{3},\frac{8}{3},\frac{8}{3})$    &  $Z_{3} = (-\frac{80}{9}, -\frac{4}{9}, \frac{44}{9})$ &  $Q_{3} =  (-\frac{375}{103},\frac{724}{103},-\frac{111}{103})$  \\ \hline

$v_4 = (-4, -3, 3)$ 				&   $X_4 = (0,5,4)$    &  $Y_{4} = (-\frac{5}{3}, \frac{5}{3}, \frac{16}{3})$    &  $Z_{4} = (0,5,3)$ &  $Q_{4} =  (-\frac{460}{11}, -\frac{395}{11}, \frac{509}{11})$  \\ \hline

\end{tabular}
}
\end{center}

\begin{center}

\renewcommand{\arraystretch}{1.9}

\resizebox{5.8in}{!} {
\begin{tabular}{|l|l|l|l|l|}

\hline

\multicolumn{5}{|c|}{Vertex Set of a 4-chromatic Subgraph of $G(\mathbb{Q}^3, \sqrt{66})$} \\ \hline

$v_0 = (0, 0, 0)$ 				&   $X_0 = (0, 2, -4)$    &  $Y_{0} = (0, 8, 8)$    &  $Z_{0} = (0, 2, -4)$ &  $Q_{0} = (\frac{829}{491},\frac{2565}{491},\frac{272}{491})$   \\ \hline

$v_1 = (4, 7, 1)$ 				&   $X_1 = (\frac{13}{3},\frac{20}{3},\frac{5}{3})$    &  $Y_{1} = (\frac{4}{93}, \frac{743}{93}, -\frac{137}{93})$    &  $Z_{1} = (-5, 4, 5)$ &  $Q_{1} =  (-7, 6, -5)$  \\ \hline

$v_2 = (-1, 11, 6)$ 				&   $X_2 = (-\frac{11}{3},\frac{25}{3},\frac{10}{3})$    &  $Y_{2} = (\frac{215}{33}, \frac{394}{33}, -\frac{229}{33})$    &  $Z_{2} = (-\frac{387}{97}, \frac{719}{97}, \frac{234}{97})$ &  $Q_{2} =  (-\frac{5354323}{3766147},-\frac{122492179}{11298441},\frac{21080450}{3766147})$  \\ \hline

$v_3 = (4, 6, 2)$ 				&   $X_3 = (-\frac{167}{307}, \frac{912}{307}, \frac{2195}{307})$    &  $Y_{3} = (-\frac{36}{5},\frac{72}{5}, 2)$    &  $Z_{3} = (\frac{380}{129}, \frac{536}{129}, \frac{530}{129})$ &  $Q_{3} =  (\frac{4643}{11837},\frac{165532}{11837},-\frac{10191}{11837})$  \\ \hline

$v_4 = (-4, 7, 1)$ 				&   $X_4 = (-\frac{4}{45}, \frac{19}{9}, \frac{353}{45})$    &  $Y_{4} = (\frac{20}{3}, \frac{5}{3}, -\frac{13}{3})$    &  $Z_{4} = (-\frac{104}{27}, \frac{193}{27}, \frac{7}{27})$ &  $Q_{4} =  (\frac{1061}{307}, -\frac{623}{307}, \frac{660}{307})$  \\ \hline

\end{tabular}
}
\end{center}

\pagebreak

Figure \ref{grotzschsubgraph2} contains the graph $H$ used as a device in Section 4, with the accompanying chart giving a representation of $H$ as a subgraph of $G(\mathbb{Q}^3, \sqrt{30})$.  Letting $S$ be the set of points simultaneously at distance $\sqrt{30}$ from each of $(-1, -2, 5)$ and $(\frac{16}{3}, \frac{8}{15}, \frac{94}{15})$, note that $S$ is a circle centered at $(\frac{13}{6}, -\frac{11}{15}, \frac{169}{30})$ and passing through point $(1,3,4)$.  It follows that $S$ has radius $\sqrt{\frac{1081}{10}}$ and Theorem \ref{vectorlength} guarantees the existence of a $4$-chromatic subgraph of $G(\mathbb{Q}^3, \sqrt{30})$.  

\begin{figure}[h]%
\begin{center}
\includegraphics[scale=.6]{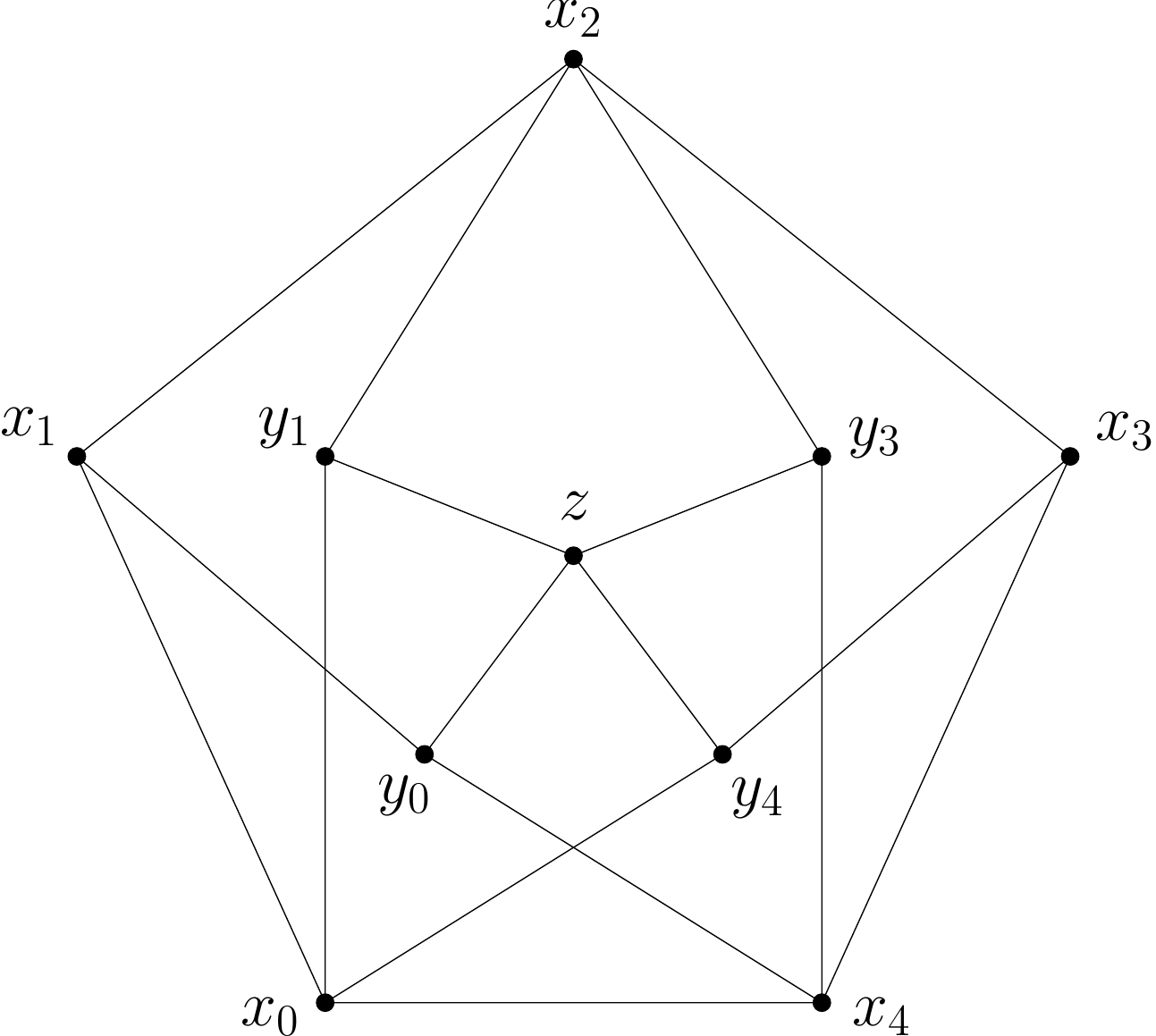}%
\caption{}%
\label{grotzschsubgraph2}%
\end{center}
\end{figure}

\begin{center}

\renewcommand{\arraystretch}{1.9}

\resizebox{5.8in}{!} {
\begin{tabular}{|l|l|l|l|}

\hline

\multicolumn{4}{|c|}{Vertex Set of a Subgraph of $G(\mathbb{Q}^3, \sqrt{30})$} \\ \hline

$x_0 = (0, 0, 0)$ 				&   $x_1 = (-1, -2, 5)$    &  $x_2 = (1, 3, 4)$    &  $x_3 = (\frac{16}{3}, \frac{8}{15}, \frac{94}{15})$    \\ \hline

$x_4 = (5, 2, 1)$ 				&   $y_0 = (-\frac{8}{21}, \frac{52}{21}, \frac{40}{21})$    &  $y_1 = (\frac{146}{63}, -\frac{145}{63}, \frac{277}{63})$    & $y_3 = (\frac{74}{21}, -\frac{191}{105}, \frac{487}{105})$     \\ \hline

$y_4 = (\frac{187}{63}, \frac{1202}{315}, \frac{811}{315})$ 				&   $z = (\frac{37}{11}, -\frac{31}{55}, -\frac{38}{55})$    &      &      \\ \hline

\end{tabular}
}
\end{center}


\begin{thebibliography}{99}

\bibitem{bau} Sheng Bau, Peter Johnson, and Matt Noble, On single-distance graphs on the rational points in Euclidean spaces, \textit{Can. Math. Bull.} \textbf{64} (1) (2020), 13 -- 24.

\bibitem{bendaperles} Miro Benda and Micha Perles, Colorings of metric spaces, \textit{Geombinatorics} \textbf{9} (3) (2000), 113 -- 126.

\bibitem{gaston} Gaston A. Brouwer, Jonathan Joe, and Matt Noble, Odd vector cycles in $\mathbb{Z}^m$, preprint.


\bibitem{burkert} Jeffrey Burkert, Explicit colorings of $\mathbb{Z}^3$ and $\mathbb{Z}^4$ with four colors to forbid arbitrary distances.  \emph{Geombinatorics} 13 (2009), no. 4, pp. 149-152.

\bibitem{chow} T. Chow,  Distances forbidden by two-colorings of $\mathbb{Q}^3$ and $A_n$,  \textit{Discrete Math.} \textbf{115} (1993), 95 -- 102.


\bibitem{ionascu1} E. J. Ionascu, A parametrization of equilateral triangles having integer coordinates, \textit{J. Integer Seq.}, \textbf{10} (2007), \#07.6.7.

\bibitem{bendaperleshistory} P. D. Johnson, Jr., Introduction to ``Colorings of Metric Spaces" by Benda and Perles  \textit{Geombinatorics} \textbf{9} (3) (2000), 110 -- 112.

\bibitem{twocolors} Peter D. Johnson Jr., Two-colorings of a dense subgroup of $\mathbb{Q}^n$ that forbid many distances, \textit{Discrete Math.} \textbf{79} (1989/1990), 191 -- 195.

\bibitem{jst} Peter Johnson, Andrew Schneider, and Michael Tiemeyer, $B_1(\mathbb{Q}^3) = 4$. \textit{Geombinatorics} \textbf{16} (1) (2007), 356 -- 362.


\bibitem{mann} Matthias Mann, Hunting unit-distance graphs in rational $n$-spaces,  \textit{Geombinatorics} \textbf{13} (2) (2003), 86 -- 97.


\bibitem{nagell} Trygve Nagell, \emph{Introduction to Number Theory}, John Wiley \& Sons, Inc., New York, 1951.

\bibitem{noble2} Matt Noble, Isosceles triangles in $\mathbb{Q}^3$,  \textit{Integers} \textbf{18} (2018), Article 9.

\bibitem{noble1} Matt Noble, On 4-chromatic subgraphs of $G(\mathbb{Q}^3, d)$, \textit{Australa. J. Combin.} \textbf{65} (1) (2016), 59 -- 70.

\bibitem{soifer} Alexander Soifer, \textit{The Mathematical Coloring Book}, Springer, 2009.

\bibitem{woodall} Douglas R. Woodall, Distances realized by sets covering the plane, \textit{J. Comb. Theory} \textbf{14} (1973), 187 -- 200.


\end{thebibliography}
\end{document}